\documentclass[12pt]{article}
\usepackage{amsmath,amsthm,amssymb,latexsym,color}
\usepackage{bbm}

\usepackage{graphicx}
\usepackage{tikz}
\usetikzlibrary{decorations.pathreplacing,shapes.misc}

\usepackage{pgfplots}

\usepackage{hyperref}

\date{}

\theoremstyle{plain}

\newtheorem*{theorem*}{Theorem}

\newtheorem{theor}{Theorem}[section]

\newtheorem*{theor*}{Theorem}
\newtheorem*{conj*}{Conjecture}

\newtheorem{thm}[theor]{Theorem}

\newtheorem{prop}[theor]{Proposition}

\newtheorem{definition}[theor]{Definition}

\newtheorem{lemma}[theor]{Lemma}

\theoremstyle{remark}

\newcommand{\B}{{\mathcal{B}}}

\def\Bo{{\mathcal B}}

\def\Be{\Bo_{d}(\eps)}

\def\Bep{\wt \Bo_{d}(\eps)}

\def\Rc{{\mathcal R}}
\def\Nc{{\mathcal N}}

\def\N{{\mathbb N}}

\def\r{\right}

\def\wt{\widetilde}

\def\disp{{\rm disp}}
\def\d{{\rm disp}}

\newcommand{\ga}{\gamma}
\newcommand{\de}{\delta}

\newcommand{\eps}{\varepsilon}

\def\qand{\quad \mbox{ and } \quad}

\title{Minimal dispersion on the cube and the torus}

\author{
A. Arman
\and
A. E. Litvak 
}
\newcommand\address{\noindent\leavevmode

\medskip

\noindent
Andrii Arman\\
Department of Mathematics,\\
University of Manitoba, \\
Winnipeg, MB, R3T 2N2, Canada\\
\texttt{\small
e-mail:   andrew0arman@gmail.com}\\

\medskip

\noindent
Alexander E. Litvak\\
Dept.~of Math.~and Stat.~Sciences,\\
University of Alberta, \\
Edmonton, AB, Canada, T6G 2G1.\\
\texttt{\small
e-mail:  aelitvak@gmail.com}\\
}

\begin{document}

\maketitle

\begin{abstract}
 We improve some upper bounds for minimal dispersion on the cube and torus.
 Our new ingredient is an improvement of a probabilistic lemma used to obtain upper bounds
   for dispersion in several previous works. Our new lemma combines a random and non-random
 choice of points in the cube. This leads to better upper bounds for the minimal
 dispersion.
\end{abstract}

\bigskip

{\small
\noindent{\bf AMS 2010 Classification:}
primary: 52B55, 52A23;
secondary: 68Q25, 65Y20.\\
\noindent
{\bf Keywords:} complexity, dispersion, largest empty box, torus


\section{Introduction}
\label{intro}

 For a given set of points $P\subset Q_d:=[0,1]^d$ its dispersion is defined as the
 supremum over volumes of axis-parallel boxes in $Q_d$ that do not intersect $P$.
 Then the minimal dispersion on the cube is defined as the infimum of the dispersions of all possible subsets
 $P\subset Q_d$ of cardinality $n$. The dispersion on the torus is defined similarly.
This notion goes back to \cite{RT}, where  a notion from \cite{Hl} was modified.
Often it is more convenient to work with its inverse function, which, given a positive $\eps$,
measures the smallest positive integer $N=N(\eps, d)$ such that there exists a configuration $X$ of $N$ points in $[0,1]^d$ with the property
that any axis-parallel box of volume exceeding $\eps$ contains at least one point.
We refer to \cite{LL} and references therein for the history of the question and related references.
%
%

 In this note we improve some upper bounds for the minimal dispersions on the cube and on the torus and
 for its inverse function. 
 
 Several recent proofs on upper bounds on $N$ use the following scheme. In the first step, approximate the set of axis-parallel boxes in $Q_d$ by a finite set $\mathcal{N}$, with the property that every axis-parallel box of volume at least $\varepsilon$ contains at least one element of $\mathcal{N}$. In the second step, construct a set $P$ such that each element of $\mathcal{N}$ intersects $P$, in which case we say that $P$ is a piercing set for $\mathcal{N}$. To ensure that $P$ pierces $\mathcal{N}$, the points in $P$ are taken randomly, according to the uniform distribution on the cube, and subsequently a union bound is applied to verify that such random $P$ with high probability intersects every element of $\mathcal{N}$.

 Our main new ingredient is in the improvement of the second step, where a random choice of points is followed up by a deterministic phase. More precisely, our new probabilistic Lemma~\ref{lemma:prob}
   will have two phases. In the first phase we choose a (smaller) set $Q$ of random points and estimate how
   many sets in our approximation $\mathcal{N}$ have an empty intersection with $Q$ (in previous proofs one insisted that every set in $\mathcal{N}$ intersects $Q$). In a second phase we take care of  the ``empty'' sets in $\mathcal{N}$ (the ones that do not intersect $Q$) by choosing a representative point
  in each such set. The set $Q_0$ of representative points together with $Q$ forms a piercing set for $\mathcal{N}$.

  The new Lemma~\ref{lemma:prob} subsequently leads to better upper bounds on minimal dispersion, see~(\ref{eq:dispbounds}). However in doing the second step we lose randomness,
  so our new bound does not hold for a random choice of points. Moreover, it is known that
  our bound cannot hold in the random setting, for instance see~\cite{HKKR},
  where the authors provide lower bounds on the inverse of minimal dispersion for a random set of points
  (formula (\ref{lrb}) below). 
  
  We would also like to mention that the idea of choosing a random set of points 
  in a first phase and then completing it by a deterministic choice (depending on a realization of a random set) in the second phase 
  is not new and was used in literature, see e.g. \cite{Rogers}, or \cite[Chapter 3]{AS}.


\section{Notation, preliminaries and main result}
\label{notat}

Throughout the paper the following notation is used. For positive integer $d$ define the unit cube $Q_d=[0,1]^d$. We will use $|\cdot|$ to denote both volume of a subset of $\mathbb{R}^d$ and a cardinality of a finite set. Letters $C, C_0, C_1$, ..., $c$, $c_0$, $c_1$, ...  always denote absolute positive constant
(independent of all other parameters and sets), whose values may change from line to line.

Define a set of all axis-parallel rectangles in $Q_d$ by
\begin{equation*}\label{eq:Rd}
	\Rc_d=\left\{\prod_{i=1}^{d} I_i \,\, | \,\, I_i=[a_i,b_i) \subset [0,1]\r\}.
\end{equation*}
For a set $P\subset Q_d$ define its dispersion by
$$\disp(P)=\sup\left\{|B| \,\, | \,\, B \in \Rc_d, \; B\cap P=\emptyset\r\}.$$
The minimal dispersion $\disp^*$ is the following function of positive integers $n,d$,
$$\disp^*(n,d)=\inf\{\disp(P) \,\,|\,\, P \subset Q_d, \; |P|=n\}.$$
Finally, we define an inverse of minimal dispersion as
$$N(\eps, d)=\min\{n \in \mathbb{N} \,\,|\,\, \disp^*(n,d)\leq \eps\}.$$
Most of the results will be stated in terms of the inverse of minimal dispersion.

We also will be dealing with the dispersion on the torus, which is defined
similarly.  Let
\begin{equation*}\label{setperrd}
  \wt \Rc _d :=\left\{ \prod _{i=1}^d I_i(a, b) \,\,\, |\,\,\,  a,  b\in  Q_d \r\},
\end{equation*}
 where
$$
   I_i(a, b) :=
   \begin{cases}(a_i, b_i), &\mbox{ whenever }\, 0\leq a_i< b_i\leq 1 ,\\
    [0, 1]\setminus  [b_i, a_i], &\mbox{ whenever }\, 0\leq b_i< a_i\leq 1. \end{cases}
$$
Then we follow definitions above with boxes taken from $\wt \Rc _d$ instead of $\Rc _d$,
that is
$$
   \wt \d (P) = \sup \{|B| \, \, \, | \, \, \, B\in \wt \Rc _d, \, B\cap P=\emptyset\},
 \quad \quad  \quad
   \wt \d^* (n, d) = \inf_{|P|=n} \wt \d(P)  ,
$$
and
$$
 \wt N(\eps, d) = \min\{ n\in \N\,  | \,\,  \wt \d^* (n, d)\leq \eps\}  .
$$


The behavior of dispersion has been intensively studied during the last decade, and it turns out that it behaves differently in different regimes.
When $\eps$ is extremely small with respect to $d$, namely $\eps < d^{-d^2}$, the best
known bound was obtained in \cite{BC}, where the authors proved that
\begin{equation*}\label{bcup}
    N(\eps, d) \leq  \frac{C\, d^2 \ln d}{\eps} .
\end{equation*}
This improved the bound $C_d/\eps$ with $C_d$ being exponential in $d$, obtained
in previous works \cite{RT, DJ, AHR}. Furthermore, the authors of \cite{BC}
have also shown that
for $\eps \leq (4d)^{-d}$ one has
\begin{equation*}\label{b-c-b}
    N(\eps, d)\geq \frac{ d}{e \eps }.
\end{equation*}

On the other hand, for relatively large $\eps$, namely for $\eps> \frac{1}{d}\, \frac{\ln ^2 d}{ \ln \ln d}$,
the best upper bound was obtained in \cite{AEL}, improving previous works \cite{Sos} and \cite{UV},
\begin{equation*}\label{large-e}
    N(\eps, d)\leq \frac{C(\ln d)  \ln(1/\eps)}{ \eps^2 }.
\end{equation*}
Very recently it was shown in \cite{TVV} that this bound is sharp (up to a logarithmic factor)
 whenever $1/4\geq \eps \geq 1/4\sqrt{d}$,
more precisely in this regime one has
\begin{equation*}\label{large-e2}
    N(\eps, d)\geq \frac{C \ln d }{ \eps^2 \ln(1/\eps)}.
\end{equation*}
We would like to mention here the bound from \cite{AHR}, which holds for all $\eps<1/4$, and which was
the first (lower) bound showing that the dispersion grows with the dimension,
\begin{equation*}\label{ahbd}
    N(\eps, d)\geq \frac{ \log_2 d}{8\eps }.
\end{equation*}

We would also like to note that for large $\eps$, namely $\eps \in (1/4, 1/2)$, the upper
bound does not grow with dimension. From  \cite{KMK} (see also \cite{Sos}) we have
$$
  N(\eps, d) \leq \frac{\pi}{\sqrt{\eps -1/4}} -3.
$$
Clearly, for $\eps \geq 1/2$ one has $N(\eps, d)=1$ by taking  the point $(1/2, ..., 1/2)$.

However, for $\eps$ not so large and not so small with respect to $d$ the picture is different.
The best known bound was obtained in \cite{LL}, improving previous results from
\cite{BEHW, Rud, AEL}, namely
\begin{equation}\label{l-l}
    N(\eps, d) \leq C\,  \frac{ d\ln \ln (1/\eps) + \ln (1/\eps)}{\eps} .
\end{equation}
The proof in \cite{LL} shows that a random choice of independent points uniformly distributed
in the cube works. Moreover, it was shown in \cite{HKKR} that  using such a random choice of points,
 one cannot expect anything better than
\begin{equation}\label{lrb}
  \max \left\{ \frac{c}{\eps} \ln\left(\frac{1}{\eps}\r),\, \frac{d}{2 \eps}\r\},
\end{equation}
that is, the bound (\ref{l-l}) is sharp  up to double logarithmic factor for a random choice of points.
In this note we improve (\ref{l-l}) by eliminating the $\ln (1/\eps)$ summand. Of course, as
the previous formula shows, the improvement cannot be achieved for a random choice of points.

Before formulating our main result we briefly discuss the dispersion on the torus.
 In \cite{MU} the lower bound was obtained valid for all $\eps\in (0, 1)$,
$$
    \wt N(\eps, d)\geq \frac{d}{\eps}.
$$
 It is interesting to note that contrary to the non-periodic
case,  the lower bound is at least linear in $d$ and always grows with $d$, even for
$\eps >1/2$.
 The best known upper bound was obtained in \cite{LL}, improving previous
 bounds from \cite{AEL} and \cite{Rud}, namely
\begin{equation}\label{eq:disp_torus}
  \wt  N(\eps, d)\leq C\,  \frac{ d \ln (2d)+ \ln (e/\eps)}{\eps }.
\end{equation}
Here, the random  choice of points was also used. In this note we improve (\ref{eq:disp_torus}) to
$C d (\ln \ln (e/\eps) + \ln (2d))/\eps $ (which is always better), however
the choice of points is, once again, not random.

\smallskip
The main result of this paper is the following.

\begin{thm}\label{main}
Let $d\geq 2$ and $\eps \in (0, 1)$. Then
$$
     N(\eps, d)\leq  \frac{16 e d \ln \ln (8/\eps)}{\eps }
   \quad   \quad \mbox{ and } \quad \quad
  \wt  N(\eps, d)\leq  \frac{8 e d (\ln \ln (e/\eps) + \ln (2d))}{\eps }.
$$
 In particular, there exists an absolute constant
$C>1$ such that  for $d\geq 2$ and $n\geq 4d$ one has
$$
 \d^*(n, d)\leq  C\, \frac{d\ln \ln (n/d) }{n }
\quad   \quad \mbox{ and } \quad \quad
  \wt   \d^*(n, d)\leq  C\, \frac{d(\ln \ln (n/d) + \ln d)}{n }.
$$
\end{thm}

\bigskip

 Thus, combining bounds of Theorem~\ref{main} with previously known bounds, the current
 state of the art for the inverse minimal dispersion on the cube is
 \begin{equation}\label{eq:dispbounds}
   N(\eps, d)\leq
 \begin{cases}
     \frac{C\, \ln d }{\eps^2}\, \ln \left(\frac{1}{\eps}\r) &\mbox{ if }\, \eps \geq  \frac{\ln^2 d}{d\ln \ln (2d)},\\[1.5ex]
     \frac{C \, d }{\eps}\, \ln \ln \left(\frac{1}{\eps}\r) &\mbox{ if }\, \frac{\ln^2 d}{d\ln \ln (2 d)}\geq \eps \geq \exp{(-d^{d})} ,\\[1.5ex]
     \frac{C\, d^2 \ln d }{\eps} &\mbox{ if }\, \eps \leq  \exp{(-d^{d})},
 \end{cases}
\end{equation}
which can be summarized in Figure~\ref{fig:1}.

\bigskip
\begin{figure}[h]
$\quad \quad\quad \quad \includegraphics[width=0.8\textwidth]{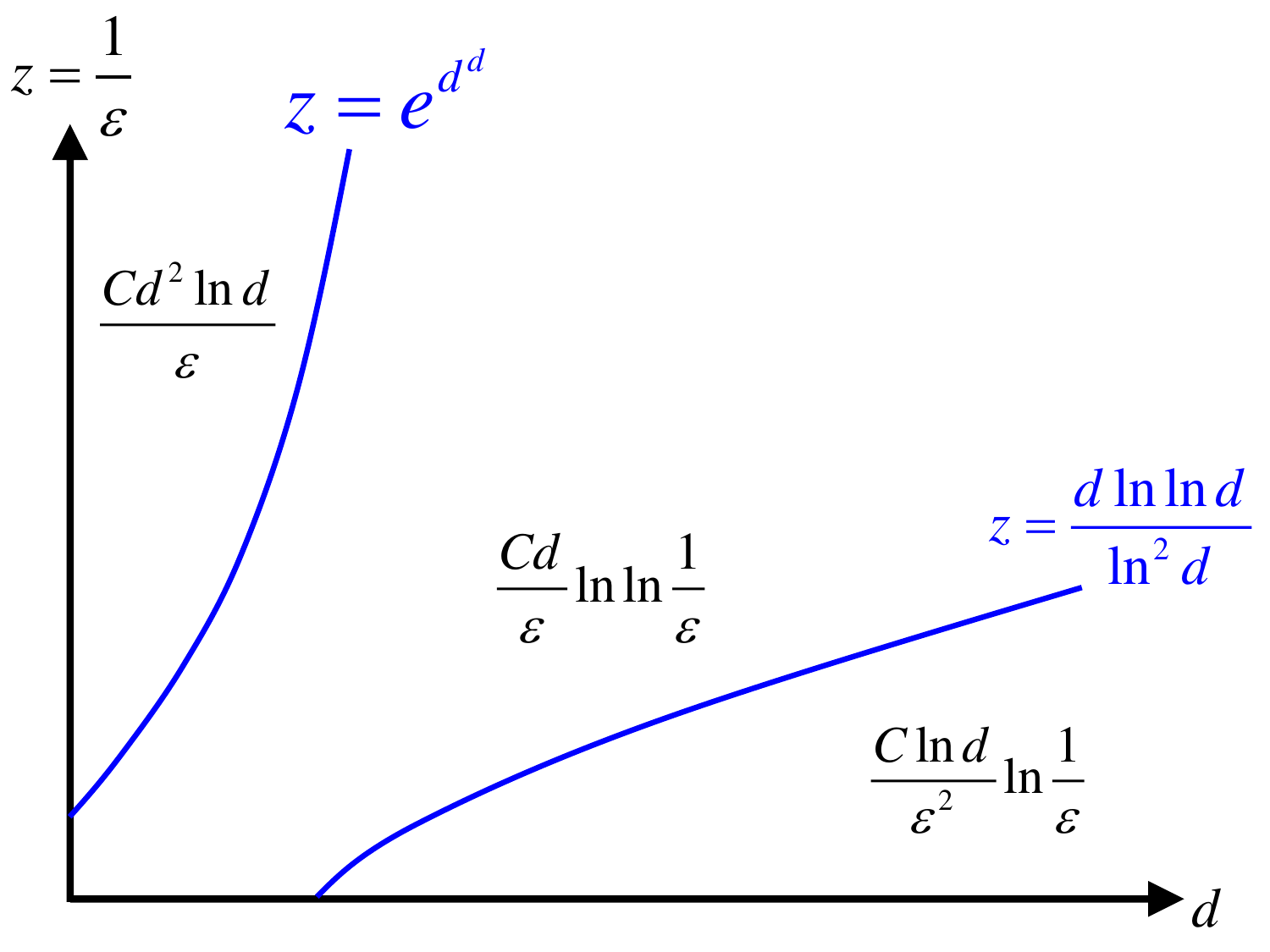}$
\caption{Current best known upper bounds on $N(\eps,d)$.}\label{fig:1}
\end{figure}

\section{A new probabilistic lemma}

For $\eps>0$ consider sets of all boxes from $\Rc_d$ (resp., $\wt \Rc _d$) of volume at least $\eps$, i.e.,
$$
   \Be :=\Big\{ B\in \Rc _d   \,\,\, |\,\,\,  |B|\geq \eps \Big\}
   \qand  \Bep :=\Big\{ B\in \wt \Rc _d   \,\,\, |\,\,\,  |B|\geq \eps \Big\}.
$$

Notice that these collections are infinite, and we use a usual approach and approximate them by finite collections. Following \cite{LL}, we define a $\delta$-approximations, which is a slight modification of the notion introduced in \cite{AEL}. An essentially same notion was recently considered in a similar context by M.~Gnewuch \cite{MGn}.

\begin{definition}[$\delta$-approximation of $\B_d(\eps)$ and $\Bep$]\label{def:deltaapprox}
	For $0<\delta\leq \eps\leq 1$ we say that a collection $\Nc\subseteq \Rc_d$ is a $\delta$-approximation for $\B_d(\eps)$ iff for all $B \in \B_{d}(\eps)$ there is $B_0\in \Nc$ such that $B_0\subseteq B$ and $|B_0|\geq \delta$. 	We define a $\delta$-approximation for  $\Bep$ in a similar way.
\end{definition}

In several works on minimal dispersion a variant of the following probabilistic
 lemma  was a key ingredient (Theorem~1 in \cite{Rud}, Lemma~2.3 with Remark~2.4
 in \cite{AEL}, Lemma~2.2 in \cite{LL}).

\begin{lemma}\label{unb}
Let $d\geq 1$ and  $\eps, \delta  \in (0,1)$. Let $\Nc$ be  a $\delta$-approximation for $\Be$
and
let   $\wt \Nc$ be a $\delta$-approximation for $\Bep$. Assume both $|\Nc|\geq 3$ and $|\wt \Nc|\geq 3$.
 Then
$$
   N(\eps, d) \leq  \frac{3\ln |\Nc|}{\delta} \quad  \qand\quad   \wt N(\eps, d)\leq \frac{3\ln |\wt \Nc|}{\delta } .
$$
Moreover, the random choice of independent points (with respect to the uniform distribution on $Q_d$)
 gives the result with probability at least  $1-1/|\Nc|$.
\end{lemma}

Our next lemma improves the bounds of Lemma~\ref{unb}.  
We would like to emphasize that our proof uses
two phases: the first random phase similar to 
the initial  proof of Lemma~\ref{unb} followed up by the second non-random phase. 
Notice that unlike Lemma~\ref{unb}, the bounds in Lemma~\ref{lemma:prob} will not hold for a random choice of independent points in $Q^{d}$, in view of (\ref{lrb}).

\begin{lemma}\label{lemma:prob}
	Let $d\geq1$, $0<\delta\leq 1/3$, $\delta \leq \eps\leq 1$.
Let $\Nc$ be  a $\delta$-approximation for $\Be$ and let   $\wt \Nc$ be a $\delta$-approximation for $\Bep$.
 Provided $\delta|\Nc |\geq e$ and $\delta|\wt \Nc |\geq e$ we have
$$
   N(\eps, d) \leq  \frac{\ln (4\delta |\Nc|)}{ \delta } \quad  \qand\quad   \wt N(\eps, d)\leq \frac{\ln(4\delta  |\wt \Nc|)}{\delta } .
$$
\end{lemma}

\begin{proof}
 We prove the bound for $N(\eps, d)$, the proof for $\wt N(\eps, d)$ is similar.
Let $|\Nc|=N$. It is sufficient to show that there is
$P\subset Q_d$ with $|P|\leq  (\ln (4\delta |\Nc|)/\delta $ such that $\disp(P)\leq\eps$. To show that $\disp(P)\leq\eps$ it is sufficient to show that for all $A\in \Nc$, we have $A\cap P \neq \emptyset$.
	
	Let $M$ be a positive  integer which will be specified later.
Consider a collection $\mathcal{X}=\{x_1,\ldots, x_M\}$ of points in $Q_d$ chosen independently and uniformly at random from $Q_d$.
	
	For each $A\in \Nc$ consider the ``bad" event $B_A=\{A \cap  \mathcal{X}=\emptyset\}$.
By assumptions on the volume of sets in $\Nc$ and by independence of $x_i$'s,  for all $A\in \Nc$ we have  $\mathbb{P}(B_A)\leq (1-\delta)^{M}$.
	
	Next let $b=b(\mathcal{X})$ be the random variable that counts the number of bad events, i.e., $b$ counts the number of sets in $\mathcal{N}$ that do not intersect $\mathcal{X}$. Clearly,
 $$b=\sum_{A\in \Nc}\chi_{B_A},$$
  where $\chi_E$ denotes the indicator of the event $E$. Then
	$$\mathbb{E}\, b=\sum_{A\in \Nc} \mathbb{E}\chi_{B_A}\leq N(1-\delta)^{M}.$$
	Therefore there is an instance $Q=\{x_1,\ldots, x_M\}\subset Q_d$ for which $b(Q)\leq N(1-\delta)^{M}$. In other words,  the set $\Nc _0:=\{A\in \Nc \; | \; A \cap Q=\emptyset\}$
	satisfies $|\Nc _0|\leq N(1-\delta)^{M}.$
	
	Finally, let $Q_0$ be a minimal collection of points which has a non-empty intersection with every set in $\Nc _0$ (we always have $|Q_0|\leq |\Nc_0|$), and define $P=Q\cup Q_0$. Then, by construction, for all $A\in \Nc$, we  have $A\cap P \neq \emptyset$.

	Note that
	$$|P|\leq M+N(1-\delta)^{M}\leq M+Ne^{-\delta M},$$
and choose $$M=\left\lceil \frac{\ln (\delta N)}{\delta}\right\rceil.$$
	Then
$$|P|\leq  \frac{1}{\delta}+  \frac{\ln (\delta N)}{\delta}+1\leq
   \frac{\ln (e^{1+\delta} \delta N)}{\delta},$$
   which completes the proof.
\end{proof}

\medskip
\noindent
{\bf Remark. }  Note that in order to have a random choice of points, i.e., in order
to show that for all $A\in \Nc$ one  has $A\cap \mathcal{X} \neq \emptyset$, one
needs to work with the event $b=0$. It is easy to see that for $M=\lfloor (3\ln |\Nc|)/\delta\rfloor$
one has
$$\mathbb{P}(b=0)\geq 1-\frac{1}{|\Nc|}.$$
This proves Lemma~\ref{unb} which  was used in \cite{Rud, AEL, LL}.

\section{Completing the proof of Theorem~\ref{main}}

To complete the proof we use the following result from \cite{LL}
(Propositions~3.2 and 3.3).

\begin{prop}\label{net-gen-box}
Let $d\geq 2$ be an integer, $\eps \in (0,1)$, and $\ga>0$. Let $\delta = \eps^{1+\ga}/4$.
There exists a $\de$-approximations $\Nc$  for $\Be$ and $\wt \Nc$  for $\Bep$
of cardinalities
$$
   |\Nc|\leq   7 d\ln d \, \frac{ (1+1/\ga)^{d }(\ln (e/\eps^{1+\ga}))^d}{\eps^{1+\ga} }
   \quad \mbox{ and } \quad
    |\wt \Nc|\leq 7 d\ln d \, \frac{ (1+1/\ga)^{d }(2d)^d}{\eps^{1+\ga} }
    .
$$
\end{prop}

The proof of Theorem~\ref{main} is a straightforward application of Lemma~\ref{lemma:prob} and
Proposition~\ref{net-gen-box}.
We provide it here for the sake of completeness.

\begin{proof}[Proof of Theorem~\ref{main}.]
Let $\gamma =1/\ln (1/\eps)$. Then $\eps ^{1+\ga}= \eps/e$. Let $\delta = \eps^{1+\ga}/4=\eps/(4e)$.
Let $\Nc$ and $\wt \Nc$ be $\delta$-approximations constructed in Proposition~\ref{net-gen-box} with
cardinalities
$$
  |\Nc|  \leq
  7 d\ln d \, \frac{ (1+1/\ga)^{d }  \, (\ln (e/\eps^{1+\ga}))^d}{\eps^{1+\ga}}
  \leq  7e d\ln d \,  \frac{(\ln (e/\eps))^{d }  \, (\ln (e^2/\eps))^d}{ \eps}
$$
and
$$
  |\wt \Nc|  \leq
  7 d\ln d \, \frac{ (1+1/\ga)^{d }(2d)^d}{\eps^{1+\ga} }
  \leq  7e d\ln d \,  \frac{(\ln (e/\eps))^{d }  \, (2d)^d}{ \eps}.
$$
Then
$$
  \ln (4 \delta |\Nc|) \leq 2d \ln \ln (e^2/\eps)+ \ln (7  d\ln d)\leq 4d \ln \ln (8/\eps)
$$
and
\begin{align*}
  \ln (4 \delta |\wt \Nc|) &\leq d \ln \ln (e/\eps)+  d\ln (2d) + \ln (7  d\ln d)
  \leq 2d\left(\ln \ln (e/\eps) + \ln (2d)\right).
\end{align*}
Hence, Lemma~\ref{lemma:prob} now implies the result.
\end{proof}

\subsection*{Acknowledgments}
\thanks{The first named author was supported by the University of Manitoba Research Grant Program project ``Covering problems with applications in Euclidean Ramsey theory and convex geometry'' and PIMS PDF.
The second named author gratefully acknowledge the support of the Leibniz Center for Informatics,
where several discussions about this research were held during the Dagstuhl Seminar
``Algorithms and Complexity for Continuous Problems'' (Seminar ID 23351).}
Part of this work was conducted during the first named author's visit to University of Alberta.


\address

\end{document}